\documentclass[12pt,a4paper]{article}
\usepackage{cite}
\usepackage{url}
\usepackage{amsmath}
\usepackage{amssymb}
\usepackage{amsthm}

\newtheorem{theorem}{Theorem}
\newtheorem{proposition}{Proposition}
\newtheorem{lemma}{Lemma}
\newtheorem{cor}{Corollary}

\newtheorem{ex}{Example}
\newtheorem{definition}{Definition}

\newcommand{\lsr}{{\check\rho}}
\newcommand{\trace}{{\mbox{trace}}}
\newcommand{\jsr}{{ \rho}}
\newcommand{\co}{{\mbox{Conv}}}

\newcommand{\re}{{\mathbb R}}
\newcommand{\ren}{{\mathbb R}^n}

\newcommand{\n}{{\mathbb N}}

\newcommand{\kron}{{\otimes}}

\newcommand{\conv}{{{\rm Conv} }}
\newcommand{\inter}{{{\rm int} }}

\providecommand{\com[2]}{\begin{tt}[#1: #2]\end{tt}}

\title{On asymptotic properties of matrix semigroups with an invariant cone}%
\author{
Rapha\"el M. Jungers\footnote{ICTEAM Institute,
Universit\'e catholique de Louvain, 4 avenue Georges Lemaitre,
B-1348 Louvain-la-Neuve, Belgium. raphael.jungers@uclouvain.be. R.
Jungers is a FNRS fellow. } }

\date{October 20, 2011}
\begin{document}

\maketitle

\begin{abstract} 
Recently, several research efforts showed that the analysis of joint spectral characteristics of sets of matrices is greatly eased when these matrices share an invariant cone.
In this short note we prove two new results in this direction.

We prove that the joint spectral subradius is continuous in the neighborhood of sets of matrices that leave an embedded pair of cones invariant.   

We show that the (averaged) maximal spectral radius, as well as the maximal trace, of products of length $t,$ converge towards the joint spectral radius when the matrices share an invariant cone, and additionally one of them is primitive. 
\end{abstract}

\section{Introduction}

The Perron-Frobenius theorem is one of the most well known theorems in linear algebra.  It states strong properties that matrices with nonnegative entries enjoy.  Starting in the fifties, is has become clear that not only this theorem, but many other properties of nonnegative matrices can be generalized to much more general cases, the fundamental feature being of importance for a matrix $A$ is that there exists an \emph{invariant proper cone} $K,$ that is, a proper cone such that $AK\subset K.$  (A cone $K$ is said \emph{proper} if it is closed, convex, with nonempty interior, and contains no straight line.  Unless explicitly mentioned, all cones are supposed to be proper in the following.)
If $K$ is an invariant cone for $A,$ we say that $A$ is \emph{$K$-nonnegative}.  If $AK \subset \inter K,$ we say that $A$ is \emph{$K$-positive.} Finally if there exists a natural number $t$ such that $A^t$ is $K$-positive, we say that $A$ is \emph{primitive}. (All these definitions are obvious generalizations of the case of nonnegative matrices with $K=\ren_+$.)
As an example of the generalizations one can obtain, let us mention the \emph{generalized Perron-Frobenius theorem} (see for instance \cite{tam-schneider-core}).
\begin{theorem} \label{thm-pf}\emph{Generalized Perron-Frobenius Theorem.}
Let $K$ be a proper cone.
If a matrix $P$ is $K$-primitive, then it has a single eigenvalue of largest modulus, which moreover is a real positive number.
\end{theorem}
The interested reader can find a survey of properties of $K$-nonnegative matrices in \cite{schneidertam}.

Our goal in this short note is to exploit the assumption of $K$-nonnegativity for the study of finitely generated matrix semigroups, and more precisely of \emph{joint spectral characteristics.}  The joint spectral characteristics of a set of matrices are quantities that allow to describe the asymptotic behaviour of the semigroup generated by this set, when the length of the products increases.  In this paper we will restrict our attention to two of these quantities:
\begin{definition}
For a set of matrices $\Sigma \subset \re^{n\times n},$ the \emph{joint spectral radius} $\rho(\Sigma)$ and \emph{joint spectral subradius} $\check \rho(\Sigma)$ are respectively defined as:
\begin{eqnarray}\label{eqn-def-radius}
\jsr (\Sigma)&\triangleq & \lim_{t \to
\infty}\max{\{ \|A_{i_1}\cdots A_{i_t}\|^{1/t}:A_i\in\Sigma\}},\\\label{eqn-def-subradius}
\nonumber \lsr (\Sigma)&\triangleq & \lim_{t \to \infty}\min{\{ \|A_{i_1}\cdots
A_{i_t}\|^{1/t}:A_i\in\Sigma\}}.
\end{eqnarray} \end{definition}
Both these limits exist and do not depend
on the norm chosen. 
They are natural generalizations of the notion of spectral radius of a matrix (i.e., the maximal modulus of the eigenvalues) to a \emph{set} of matrices.  
The
joint spectral radius appeared in~\cite{rota-strang}, and the joint
spectral subradius in~\cite{gu1}.  Not only these quantities account for stability issues of switching linear systems, but they have also found many applications in various different fields of Engineering, Mathematics, and Computer Science.  See~\cite{jungers_lncis} for a recent survey on these quantities.\\
The issue of \emph{computing the joint spectral radius} has also been largely studied, and several negative
results are available in the literature.  For instance, the issue of exactly computing the joint spectral radius, as well as the subradius, are both known to be undecidable.  This fact has led to a rich literature, where techniques from very different fields (from Control Theory to Ergodic Theory, Automata Theory, etc.) have been applied to the topic (see \cite{fainshil-margaliot,morris-ergodic,GZalgorithm,liberzon-cdc10,ajprhscc11,protasov1} as a few examples).\\
The joint spectral subradius has been much less studied until recently. In fact, the only methods we know of that allow to compute the joint spectral subradius have been proposed only very recently \cite{protasov-jungers-blondel09,GP11}.  These methods are proved to be efficient only in some favorable cases, namely, when the set of matrices share a \emph{common invariant embedded pair of cones} (see below for definitions). It appears that this assumption allows for more appealing properties, and in this note we prove (Section \ref{section-continuity}) that under the same hypothesis, the joint spectral subradius is a continuous function.  Then, in Section \ref{section-trace}, we study the joint spectral radius, and we generalize a recent result 
on nonnegative matrices to any set leaving a cone invariant.  This result is about the convergence of the maximum trace, and the maximum spectral radius of the products towards the joint spectral radius, when the length of the products increases.

The joint spectral characteristics have attracted much attention in recent years, especially in the case of nonnegative matrices because many applications involve such matrices \cite{valcher-metzler,jungersprotasovblondel06,gurvits-positive,fainshil-margaliot}.  These applications make the theoretical study of ($K$-)nonnegative matrices of particular importance.

\section{Continuity of the subradius}\label{section-continuity}

The continuity of the joint spectral radius (w.r.t. the Hausdorff distance\footnote{The \emph{Hausdorff distance} measures the distance between sets of points in a metric space:  $$D(\Sigma,\Sigma')\triangleq  \max{\{\sup _{A\in\Sigma}{\{\inf_{A'\in\Sigma'}{||A-A'||}\}},\sup_{A'\in\Sigma'}{\{\inf_{A\in\Sigma}{||A-A'||}\}}\}}$$}) is well-known:

\begin{proposition}\label{prop-continuity}\cite{wirth02generalized}
For any bounded set of matrices $\Sigma\in\ren,$ and for any $\epsilon>0,$ there is a $\delta>0$ such that
$$D(\Sigma,\Sigma')<\delta\Rightarrow |\rho(\Sigma)-\rho(\Sigma')|<\epsilon.$$ \end{proposition} 

Less is known on the subradius.  This quantity is not continuous, as shown on the next example, drawn from \cite{jungers_lncis}:

\begin{ex}
The sequence of sets $$ \Sigma_k = \left\{ \begin{pmatrix}1&1\\0&1\end{pmatrix},\begin{pmatrix}0&0\\-\frac{1}{k}&1\end{pmatrix} \right\}, \quad k\in \n,$$ converges towards $$ \Sigma = \left\{ \begin{pmatrix}1&1\\0&1\end{pmatrix},\begin{pmatrix}0&0\\0&1\end{pmatrix} \right\}$$
when $k \rightarrow \infty.$ For any $k\in \n,$ we have $\check \rho (\Sigma_k)=0$ because the product $(A_1A_0^k)^2$ is the zero matrix.  However $\check\rho(\Sigma)=1$ (This is because the lower right entry of any product is equal to one) and thus, the joint spectral subradius is not continuous in the neighborhood of $\Sigma.$
\end{ex}
Thus, the nonnegativity of matrices is not sufficient for ensuring the continuity of the subradius in the neighborhood of the set.  However, we show in this section that if the matrices share a second invariant cone, then we have the continuity.
More precisely, let us consider a cone $K   \subset
\re^n.$ We say that a convex closed cone $K'$ is {\em embedded in} $K$ if
$(K'\setminus \{0\}) \subset  \inter K.$ In this case, following \cite{protasov-jungers-blondel09}, we call $\{K, K'\}$ an
{\em embedded pair}. Note that  the embedded
cone $K'$ may be degenerate, i.e., may have an empty interior. An embedded pair $\{K,K'  \}$ is called
an {\em invariant pair} for a matrix
$A$ (or a set of matrices $\Sigma$) if the cones $K$ and  $K'$ are
both invariant for $A$ (for the matrices in $\Sigma$). The following definition aims at characterizing the ``embeddedness'' of the pair $(K,K').$
\begin{definition}\label{def-embedded}\cite{protasov-jungers-blondel09}
For a given embedded pair $\{K, K'\}$ the value $\beta (K, K')$ is the smallest number
such that for any line intersecting $K$ and $K'$ by segments $ [x,y] $ and $ [x', y'] $ respectively
(with $ [x,x']\subset [x, y'] $) one has
$1\leq \frac{|x-y'|}{|x-x'|} \le  \beta .$
\end{definition}
In the following we denote by $\Sigma^t$ the set of products of length $t$ of matrices in $\Sigma.$
In the developments below we use the two following results:

\begin{lemma} \cite{jungers-protasov-blondel07}\label{lem-conic}
Let $\Sigma$ be a compact set of matrices that share an invariant cone $K.$  If there exists a nonzero vector $x\in K$ such that $$ \forall A\in \Sigma, Ax\geq_K r x,  $$ then $\check \rho (\Sigma)\geq r.$
\end{lemma}
\begin{theorem}\label{th-beta}\cite[Theorem 2.12]{protasov-jungers-blondel09}
For any compact set $\Sigma$ with an invariant pair  $\{K, K'\},$ there exists a nonzero vector $x\in K'$ such that $$ \forall A\in \Sigma, Ax\geq_K \check{\rho}(\Sigma) x/\beta,  $$where $\beta=\beta(\{K, K'\}).$
\end{theorem}

We are now in position to prove the main theorem of this section:
\begin{theorem}\label{main-thm}
Let $\Sigma$ be a compact set of matrices in $\mathbb{R}^{n\times n},$ and let $\Sigma_{k}$ be a sequence of sets in $\mathbb{R}^{n\times n}$ that converges to $\Sigma$ in the Hausdorff metric.  If $\Sigma$ leaves an embedded pair of cones invariant, then, $$ \check \rho(\Sigma_k)\ \rightarrow  \check \rho(\Sigma) \quad \mbox{ as $k \rightarrow \infty$}. $$
\end{theorem}

\begin{proof}

  Let us consider a set $\Sigma$ that leaves an embedded pair of cones $(K,K')$ invariant, such that $\beta(K,K')=\beta.$  If $\check\rho(\Sigma)=0$ then the joint spectral subradius is continuous at $\Sigma,$ being upper semicontinuous and nonnegative in general \cite{jungers_lncis}.  We can then suppose that $\check \rho (\Sigma)>0$ and, by scaling the set of matrices, we suppose that  $\check\rho(\Sigma)=2$ (the joint spectral subradius is an homogeneous function of the entries of the matrices).

Fix $\epsilon>0.$  By upper semicontinuity, we know that there exists a $\delta >0 $ such that $$D(\Sigma',\Sigma)<\delta\quad \Rightarrow \quad \check \rho (\Sigma')<\check\rho (\Sigma)+\epsilon.$$ We still have to show that for some $\delta'>0,$ if $D(\Sigma',\Sigma)<\delta',$ then $$\check \rho (\Sigma')>\check\rho (\Sigma)-\epsilon.$$   
It follows directly from the Definition (\ref{eqn-def-subradius}) that $$\check\rho(\Sigma^t)=\check\rho^t(\Sigma). $$ Combining this with Theorem \ref{th-beta}, there must exist an integer $t$ and a vector $x\in K',$ $|x|=1,$
 such that $$\forall A\in \Sigma^t,\quad Ax\geq_{K'}(2-\epsilon/2)^tx.$$ Take also $t$ large enough such
 that $(2-\epsilon/2)^t\geq (2-\epsilon)^t+\epsilon/2.$  Since $x\in K'\subset \inter(K),$
  we can define $\eta>0$ such that for any vector $y:$ $$|y|<\eta,\quad \Rightarrow \quad (\epsilon/2)x+y \in K.$$
 Then, $\Sigma$ being compact, there exists a $\delta' >0$ such that $$ D(\Sigma',\Sigma)<\delta' \quad \Rightarrow D(\Sigma'^t,\Sigma^t)<\eta.$$
Take such a set of matrices $\Sigma',$ so that $D(\Sigma'^t,\Sigma^t)<\eta.$    Then, for any product $A'\in \Sigma'^t$, taking $A$ the corresponding product in $\Sigma^t$ close to $A',$ 
\begin{eqnarray}\nonumber A'x &=& (A'-A+A)x\\\nonumber &\geq_{K}&(2-\epsilon/2)^tx +(A'-A)x,\\\nonumber &\geq_{K}&(2-\epsilon)^tx +(A'-A)x+(\epsilon/2)x\\\nonumber  &\geq_K &(2-\epsilon)^tx . \end{eqnarray}

By Lemma \ref{lem-conic}, this implies that $\check \rho (\Sigma')\geq 2-\epsilon.$
\end{proof}

\begin{cor}
The joint spectral subradius is continuous in the neighborhood of any set of matrices with positive entries.
\end{cor}
\begin{proof}
It is known (see \cite[Corollary 2.14]{protasov-jungers-blondel09}) that if all matrices of a set $\Sigma$ are positive and in each column of any matrix the ratio between the
greatest and the smallest elements
  does not exceed $c,$ then $\Sigma$ has an invariant pair of embedded cones, for which $ \beta \le  c^{ 2}.$
\end{proof}

\section{Asymptotic regularity}\label{section-trace}

It is well known that for any bounded set of matrices, the joint spectral radius can be alternatively defined in terms of the maximal spectral radius (instead of the maximal norm):
\begin{theorem}(\textbf{Joint Spectral Radius Theorem}, \cite{berger-wang})\label{thm-fond-jsr}\\
For any bounded set of matrices $\Sigma,$
$$\limsup_{t\rightarrow \infty}{\sup{\{\rho^{1/t}(A):A\in\Sigma^t\}}}=\rho(\Sigma). $$\end{theorem}

Some attention has been recently given to a peculiarity in the alternative definition provided by Theorem \ref{thm-fond-jsr}: the maximal averaged spectral radius asymptotically converges towards the value of the joint spectral radius in limit superior, but not at every time steps.  Simple examples illustrate this fact, as for instance \begin{equation}
\Sigma=\left \{  \begin{pmatrix}

0& 1\\
0&0\end{pmatrix}
         ,\
\begin{pmatrix}

  0&0\\1 & 0
\end{pmatrix}\right \}.
\end{equation}

For this set of matrices $\rho(\Sigma)=1$ but all products of odd length have a spectral radius equal to zero.
It is thus natural to try to understand when the maximal averaged spectral radius actually converges or to give sufficient conditions for it. 

%

Moreover, the next proposition shows that the joint spectral radius can also be defined as the limit superior of the rate of growth of the \emph{maximal trace} in the semigroup:

\begin{proposition}\cite{chen-zhou-characterization}
For any finite set of matrices, the joint spectral radius satisfies
\begin{equation}
\rho(\Sigma)=\limsup_{t \rightarrow \infty}{\max_{A\in\Sigma^t}{\{\trace^{1/t}(A)\}}}.
\end{equation}
\end{proposition}
Here again, the joint spectral radius is defined as a limit superior, and one could wonder when the sequence actually converges.
A sufficient condition relying on the nonnegativity of the matrices has been proposed recently:
\begin{theorem}\label{thm-xu}\cite{xu-ela}
Let $\Sigma=\{A_1,\dots, A_n\}$ be a set of nonnegative matrices.  If one of them is primitive, then \begin{equation}\label{eq-thm-xu-thm}
\max_{A\in\Sigma^t}{\{\trace(A)^{1/t}\}} \quad \rightarrow \quad \rho(\Sigma)
\end{equation}
as $t\rightarrow \infty.$
\end{theorem}

In this section we generalize Theorem
\ref{thm-xu} to arbitrary invariant cones:

\begin{theorem}\label{thm-regularity}
Let a set of matrices $\Sigma=\{A_1,\dots,A_m\}$ share an invariant cone $K.$   If there is a matrix $A_i\in \Sigma$ which is \emph{$K$-primitive}, then, both the quantity
  $$
\max_{A\in\Sigma^t}{\{\trace^{1/t}(A)\}}$$ and $$
\max_{A\in\Sigma^t}{\{\rho^{1/t}(A)\}}$$
converge towards $\rho$ when $t$ tends to $\infty.$
\end{theorem}

The proof 
is inspired from developments in \cite{xu-ela} and results from \cite{blondel-kron} which we now recall.

\begin{definition}\textbf{Kronecker product.}\index{Kronecker!product}\index{Kronecker!power}
Let $A\in\re^{n_1\times n_2},\ B\in \re^{m_1 \times m_2}.$ The \emph{Kronecker product} of $A$ and $B$ is a matrix in $\re^{n_1m_1\times n_2m_2}$ defined as
$$(A \kron B) \triangleq \begin{pmatrix} A_{1,1}B&\dots & A_{1,n_2}B \\ \vdots &\vdots &\vdots \\ A_{n_1,1}B&\dots & A_{n_1,n_2}B \end{pmatrix}.$$
The \emph{$k$-th Kronecker power of $A$,} denoted $A^{\kron k},$ is defined inductively as
$$A^{{\kron k}}=A\kron A^{\kron (k-1)} \quad A^{\kron 1}=A.
$$
\end{definition}

\begin{proposition}\label{prop-kron}
\begin{enumerate}
\item For any matrices $A,B,C,D,$ $$(A\otimes B)(C\otimes D)=(AC)\otimes (BD). $$
\item For any matrix $A,$ we have $$ \trace{(A^{\otimes k})}=\trace^k{(A)}.$$ \end{enumerate}
\end{proposition}

The main reason for introducing Kronecker products is the following theorem:
\begin{theorem}\label{thm-kron-jsr}\cite{blondel-kron}
Let $\Sigma=\{A_1,\dots, A_m\}$ be a set of matrices that leaves a cone invariant.  Then, we have the following property:
\begin{equation}
\rho(\Sigma)=\lim_{k\rightarrow \infty}{\rho^{1/k}(A_1^{\otimes k}+\cdots + A_m^{\otimes k})}.
\end{equation}
\end{theorem}

Before to prove Theorem \ref{thm-regularity}, we prove a last technical lemma that will be necessary in the proof.
\begin{lemma}\label{lem-prim-kron}
Let $K$ be a proper cone.  If a matrix $A$ is $K$-primitive, then for any $k,$ $A^{\otimes k}$ is $K^{\otimes k}$-primitive,where $$K^{\otimes k}\triangleq \conv\{x_1\otimes\cdots\otimes x_k:x_i\in K,\ i=1,\dots, k \}$$ is a proper cone. 
\end{lemma}
\begin{proof}
The fact that $K^{\otimes k}$ is a proper cone is proved in \cite[Lemma 4]{blondel-kron}.

It is clear that if $A$ is $K$-nonnegative, then $A^{\otimes k}$ is $K^{\otimes k}$-nonnegative, since $$A^{\otimes k} (x_1\otimes \dots\otimes x_k)= Ax_1\otimes\dots\otimes Ax_k. $$
It remains to show that all extremal points of $K^{\otimes k}$ are mapped in the interior of $K^{\otimes k}.$\\
So, let us take an extremal point $\tilde x= x_1\otimes \dots \otimes x_k$ of $K^{\otimes k},$ and show that for some $t>0,$ $(A^{\otimes k})^t \tilde x \in \inter K^{\otimes k}.$  Since $$(A^{\otimes k})^t \tilde x = A^tx_1\otimes \dots \otimes A^tx_k$$ and $A$ is primitive, it suffices to show that $$x_i\in \inter K\quad \Rightarrow \quad x_1\otimes \dots \otimes x_k\in \inter K^{\otimes k}. $$\\
We will in fact show the slightly stronger property that for any set of proper cones $K_i,$ $$x_i\in \inter K_i\quad \Rightarrow \quad x_1\otimes \dots \otimes x_k\in \inter (K_1\otimes \dots \otimes K_k), $$ where $$K_1\otimes \dots \otimes K_k\triangleq \co{\{x_1\otimes \dots \otimes x_k:x_i\in K_i\}}.$$  By induction, it is sufficient to prove it for $k=2.$
Let us be given two proper cones $K_1\in \re^{n_1},\, K_2\in \re^{n_2},$ $x_i\in \inter(K_i).$ Suppose by contradiction that $ x_1\otimes x_2 \in K_1\otimes K_2 \setminus \inter (K_1\otimes K_2).$  Then, there exists a vector $z\in \re^{n_1\times n_2},z\neq 0,$ such that $(z, x_1\otimes x_2)=0,\, (z,w)\geq 0\,  \forall w\in K_1\otimes K_2. $  Now, since for any $i,$ $1\leq i \leq n_1,$ there exists a $\lambda$ small enough such that $ (x_1+\lambda e_i)\otimes x_2\in K_1{\otimes K_2}$ ($e_i$ is the $i$th standard basis vector); this implies that $$ (z,(x_1+\lambda e_i))\otimes x_2 =(z,x_1\otimes x_2) +(z,\lambda e_i\otimes x_2)= (z,\lambda e_i\otimes x_2) \geq 0 ,$$ and thus  $(z, e_i\otimes x_2) =0$ (because $\lambda$ can be positive and negative). For the same reason, for any $j,$ $(z, x_1\otimes e_j) =0.$ Now, \begin{eqnarray} \nonumber (z,(x_1+\lambda e_i)\otimes (x_2+\gamma e_j) &=&(z,x_1\otimes x_2) +(z,\lambda e_i\otimes x_2)\\ \nonumber &&+(z, x_1\otimes \gamma e_j)+(z,\lambda \gamma e_i\otimes e_j)\\ \nonumber &=&(z,\lambda \gamma e_i\otimes e_j)\\ \nonumber & \geq & 0.\end{eqnarray}Thus, $(z,e_i\otimes e_j)=0$ for all $i,j,$ which implies that $z=0,$ a contradiction.
\end{proof}

We are now in position to prove the main result of this section.

\begin{proof} (Proof of Theorem \ref{thm-regularity})
We first \textbf{claim} that for any natural number $k,$ the $k$th Kronecker powers of the matrices in $\Sigma$ satisfy the property $$\lim_{t\rightarrow \infty}{\trace^{1/t}{(A_1^{\otimes k}+ \cdots + A_m^{\otimes k})^t}}=\rho(A_1^{\otimes k}+ \cdots + A_m^{\otimes k}).$$\\
By Lemma \ref{lem-prim-kron} above, one of the matrices $A_i^{\otimes k}$ is $K^{\otimes k}$-primitive. Moreover, it is well known that a sum of nonnegative matrices, one of which is primitive, is actually also primitive, thus, so is $$A_1^{\otimes k}+ \cdots + A_m^{\otimes k}.$$  Finally, by Theorem \ref{thm-pf}, it is clear that $$\lim_{t\rightarrow \infty} (\trace{(P^t)}^{1/t}) =\rho(P),$$ and the claim is proved.

Now, remark that for any $k,t\in \n,$  \begin{equation}\trace{((\sum_{A \in \Sigma } A^{\otimes k})^t)}=\sum_{A\in \Sigma^t}\trace{(A^{\otimes k})} \leq m^t (\max_{A\in \Sigma^t}{\trace^k(A)}).\end{equation}
For deriving the above relations, we successively used items 1 and 2 of Proposition \ref{prop-kron}. Dividing this equation by $m^t$ and taking to the power $1/(kt),$ we obtain 
\begin{equation}\label{eq-thm-xu}\trace^{1/(t k)}{((\sum_{A \in \Sigma } A^{\otimes k}/m)^t)}\leq  \max_{A\in \Sigma^t}{\{\trace^{1/t}(A)\}}.\end{equation}  Now, for any $k\in \n,$ the lefthand side in (\ref{eq-thm-xu}) tends towards $\rho^{1/k}(\sum_{A \in \Sigma } A^{\otimes k}/m)$ as $t\rightarrow \infty.$\\ This proves part one of the theorem, as the latter quantity is arbitrarily close to $\rho(\Sigma)$ for large $k$ (Lemma \ref{thm-kron-jsr}).

For the second part of the theorem, since $\rho(A)\geq \trace{A}/n$ ($n$ is the dimension of the matrices), we have $$
\max_{A\in\Sigma^t}{\{(\trace(A)/n)^{1/t}\}} \leq  \max_{A\in\Sigma^t}{\{\rho(A)^{1/t}\}},$$ which immediately implies that $\max_{A\in\Sigma^t}\{\rho^{1/t}(A)\}\rightarrow \rho$ when $t\rightarrow \infty.$
%

\end{proof}

\newcommand{\noopsort}[1]{} 
\newcommand{\singleletter}[1]{#1}

\end{document}